\documentclass{article}
\usepackage[english]{babel}
\usepackage[utf8]{inputenc}
\usepackage{johd}
\usepackage{mathptmx}
\usepackage{amsmath}
\usepackage{mathtools}
\usepackage{verbatim}
\usepackage{appendix}

\numberwithin{equation}{section}

\usepackage{amsmath,amsthm}       
\usepackage{mathrsfs}      
\usepackage{amssymb}       
\usepackage{amsfonts}      

\usepackage{cancel}
\usepackage{indentfirst}

\usepackage{graphicx}
\usepackage{subfigure}
\usepackage{wrapfig}  
\usepackage{multicol}
\usepackage{tikz}     

\usepackage[framemethod=TikZ]{mdframed}
\usepackage{color}
\usepackage{authblk}                          
\usepackage{textcomp}
\newtheorem{thm}{Theorem}[section]
\newtheorem{lemma}[thm]{Lemma}

\newtheorem{remark}[thm]{Remark}
\newtheorem{prop}[thm]{Proposition}
\newtheorem{example}[thm]{Example}

\author[1]{Zihao Gu}
\author[1]{Yiqing Lin}
\author[,1]{Kun Xu\footnote{Corresponding author. Email address: \url{1949101x_k@sjtu.edu.cn} (K. Xu).}}
\affil[1]{\small{School of Mathematical Sciences, Shanghai Jiao Tong University, 200240 Shanghai, China.}}

\begin{document}

\title{Mean reflected BSDE driven by a marked point process and application in insurance risk management}

\date{October 15, 2023}
\maketitle
%
%
%
\begin{abstract}

This paper aims to solve a super-hedging problem along with  insurance re-payment under running risk management constraints. The initial endowment for the super-heding problem is characterized by a class of mean reflected backward stochastic differential equation driven by a marked point process (MPP) and a Brownian motion. By Lipschitz assumptions on the generators and proper integrability on the terminal value, we give the well-posedness of this kind of BSDEs by combining a representation theorem with the fixed point argument.
\end{abstract}


\section{Introduction}
Nonlinear backward stochastic differential equations (BSDEs) driven by a Brownian motion were first introduced by Pardoux and Peng \cite{Pardoux1990AdaptedSO} in 1990. Since then, varies of applications of BSDEs have been investigated by many researchers, see for instance \cite{Karouibackward1997,MR3154654} for financial applications and \cite{MR3308895} for applicatitons in PDEs. In view of the structure generalization of standard BSDEs, the Wiener process in the diffusion term was replaced by more general martingales, see e.g. \cite{MR1752673,Carbone2008BackwardSD,cohen2012existence} for more detailed discussion.  In particular, a class of BSDEs driven by a random measure associated with a marked point process have aroused a lot of attention. We refer \cite{Bremaud1981,last1995marked, Jacobsen2006} as general references on marked point processes. In spite of these literature, BSDEs driven by a random measure has been introduced in \cite{Tang_1994,Barles_1997,Royer_2006,Kharroubi_2010}, with  applications in  stochastic maximum principle, nonlocal partial differential equations, nonlinear expectation, quasi-variational inequalities and impulse controls. Firstly, the jump part of all aforementioned BSDEs is considered as a Poisson-type random measure. General random measures beyond Poisson were considered in \cite{Xia2000}. In particular, the well-posedness of BSDEs driven by general marked point processes were investigated in \cite{Confortola2013} for the weighted-$L^2$ solution,  \cite{Becherer_2006,Confortola_2014} for the $L^2$ case,  \cite{Confortola2016} for the $L^1$ case and \cite{Confortola_2018} for the $L^p$ case.

Moreover, \cite{El_Karoui_1997} studied reflected backward stochastic differential equations (RBSDEs) in order to solve an obstacle problem for PDEs. Thereafter, generalizations on the structure of the diffusion term are followed. In addition to the Wiener process, a very general marked point process, which is non-explosive and has totally inaccessible jumps, is added to the diffusion term in \cite{foresta2021optimal}. 

In contrast to the pointwisely reflected BSDEs, mean reflected BSDEs were introduced in \cite{briand2018bsdes} when dealing with super-hedging problems under running risk management constraints. The formulation of the restricted BSDE reads,
\begin{equation}
\label{original}
\left\{\begin{array}{l}
 Y_t=\xi+\int_t^T f\left(s, Y_s, Z_s\right) d s-\int_t^T Z_s d W_s+K_T-K_t, \quad 0 \leq t \leq T, \\
\mathbb{E}\left[\ell\left(t, Y_t\right)\right] \geq 0, \quad 0 \leq t \leq T ,
\end{array}\right.
\end{equation}
where  $K$ is a deterministic process and $\ell$ is  a running loss function. The well-posedness of  such BSDEs with mean reflection was generalized in \cite{Hibon2017quadratic, hu2022general}, with quadratic generator and bounded or unbounded terminal condition. The following type of Skorokhod condition,
$$
\int_0^T \mathbb{E}\left[\ell\left(t, Y_t\right)\right] d K_t=0,
$$
ensures the existence and uniqueness of the so-called deterministic flat solution.

This paper generalizes the mean reflected BSDEs by allowing a random measure in the diffusion part. A similar formulation of mean reflected SDEs with jumps can be found in \cite{Briand_2020}. In the light of \cite{briand2018bsdes,Hibon2017quadratic,hu2022general}, we consider the following mean reflected BSDEs driven by a marked point process,
\begin{equation}
\label{eq01}
\left\{\begin{array}{l}
Y_t =\xi+\int_t^T f\left(s, Y_s, U_s\right) d A_s+\int_t^T g\left(s, Y_s, Z_s\right) d s\\
\quad \quad-\int_t^T \int_E U_s(e) q(d s d e)-\int_t^T Z_s d W_s + \left(K_T-K_t\right) , \quad \forall t \in[0, T] \text { a.s. };\\
\mathbb{E}\left[\ell\left(t, Y_t\right)\right] \geq 0, \quad \forall t \in[0, T] .
\end{array}\right.
\end{equation}
Here $W$ is a $d$-dimensional Brownian motion and $q$, independent from $W$, is a compensated integer random measure corresponding to some marked point process $(T_n,\zeta_n)_{n\ge 0}$. The process $A$ is the dual predictive projection of the event counting process related to the marked point process. We emphasize that under our assumptions, as in \cite{foresta2021optimal}, the process $A$ is continuous and increasing which is not necessarily absolutely continuous with respect to the Lebesgue measure.
With the help of a fixed point technique, we construct the well-posedness of (\ref{eq01}) in weighted-$L^2$ spaces with a weight of the form $e^{\beta A_t}$. The generators $f$ and $g$ are assumed being Lipschitz and the terminal condition $\xi$ is of proper integrability.

Compared with BSDEs only driven by a Brownian motion, for example in \cite{briand2018bsdes}, delicate techniques are used to build the well-posedness of mean reflected BSDEs (\ref{original}) with the discountinous setting. In particular, the spaces where the solutions live in are quite different, which bring difficulties in the construction of a contracting mapping. Eventually, we apply the fixed-point argument with the help of a priori estimates and the small time interval stitching technique.

When it comes to the financial applications of BSDEs, as in El Karoui, Peng and Quenez \cite{Karouibackward1997}, classical BSDEs can be applied to the pricing of European contingent claims. Therein, the components $Y$ and $Z$ of the solution are related to the value process of the claim and its hedging strategy, respectively. Afterwards, 
El Karoui, Pardoux and Quenez \cite{el1997reflected} studied  the price of an American option which can be formulated as the “minimal” solution to the following BSDE: 
\begin{equation}
\label{reflected}
\left\{\begin{array}{l}
Y_t=\xi+\int_t^T f\left(s, Y_s, Z_s\right) d s-\int_t^T Z_s d W_s+K_T-K_t, \quad 0 \leq t \leq T, \\
Y_t \geq L_t, \quad 0 \leq t \leq T .
\end{array}\right.
\end{equation}
where $K$ is adapted and non-decreasing. More recently, as mentioned before, unlike the pointwisely reflected BSDEs,  Briand, Elie and Hu \cite{briand2018bsdes} have formulated mean reflected BSDEs (\ref{original}) which can be used to solve super-hedging problems under running risk management constraints.

In this paper, we consider a financial application in a more general discontinuous framework. Our main contribution is to help insurance companies hedge mortality risk under a risk management constraint through solving the mean reflected BSDE (\ref{eq01}). Under our setting, each company is assumed to have such an insurance payment process:
\begin{equation}
\label{eq insurance intro}
\begin{aligned}
P_t= & \int_0^t(n-N_s) {H}_s d s+\int_0^t\int_{E}{G}_s(e)p(dsde), \quad 0 \leq t \leq T,
\end{aligned}
\end{equation}
in which $H$ is a deterministic process representing the insurance premium received continuously during the period of the contract and $G$ is a deterministic death benefit paid at random times triggered by a marked point process $p$. Moreover, the company faces a terminal payoff $(n-N_T) {F}$ at maturity $T$, where the random variable $F$ is a survival benefit paid at the end of the contract. We model the mortality by a marked point process since We take into consideration the different causes of death, such as  natural death, traffic accident, sudden illness and etc. The causes of death are marked by a marker in a finite mark space $E$.
In order to hedge the  mortality risk, mortality derivatives have been introduced in global financial markets. Thus, we assume that the
financial market consists of the bank account, the stock and a mortality bond. The insurer can use a mortality bond to hedge death and survival benefits from its portfolio since pay-offs from mortality bonds are contingent on the mortality experience in a population, see Blake et al. \cite{blake2009longevity} and Wills and Sherris \cite{wills2010securitization}. Besides, the constraint in (\ref{eq01}) can be replaced by a more general version of the form
\begin{equation}
\label{crho}
\rho(t,Y_t)\le q_t,\ 0\le t\le T,
\end{equation}
where $\{\rho(t,\cdot)\}_{0\leq t \leq T}$ is a time indexed collection of static risk measures, and $\{q_t\}_{0\leq t \leq T}$ are associated benchmark levels. Consider portfolios $X^{x, \pi,\chi, K}$ whose dynamics is given by
$$
\begin{aligned}
d X_t^{x, \pi,\chi, K} & = \pi_t(\mu_t d t+\sigma_t d W_t)+\chi_t\frac{dD_t}{D_{t^-}}+\left(X_t^{x, \pi,\chi, K}-\pi_t-\chi_t\right) r_t d t-dP_t-d K_t
\end{aligned}
$$
where $x$ is a given initial capital, $\pi$ and $\chi$ denote the amount of wealth invested in the stock $S$ and in the mortality bond $D$.
By changing the measure, eventually, we are able to  tackle the super-hedging problem by solving a BSDE with risk measure reflection.

The rest of the paper is organized as follows. In Section \ref{sec P}, we first recall some basic notations on marked point processes and some results on BSDEs driven by a marked point process. In Section \ref{sec F}, we present the problem which we are going to solve. Section \ref{sec B} is devoted to the well-posedness of mean reflected BSDEs driven by a marked point process with fixed generators. The general case in which the generators depend on the solution is discussed in Section \ref{sec ger}. We end this paper by solving a super-hedging problem with an insurance re-payment under a risk management constraint in Section \ref{sec app}. 

\section{Preliminaries}
\label{sec P}

\subsection{Notations on marked point process}

We first recall some notions about marked point processes. More details can be found in \cite{foresta2021optimal, Bremaud1981, last1995marked, cohen2012existence}. 

In this paper we assume that $(\Omega, \mathscr{F}, \mathbb{P})$ is a complete probability space and $E$ is a Borel space. We call $E$ the mark space and  $\mathscr{E}$ is its Borel $\sigma$-algebra. Given a sequence of random variables $(T_n,\zeta_n)$ taking values in $[0,\infty]\times E$, set $T_0=0$ and $\mathbb P-a.s.$ 
\begin{itemize}
\item $T_n\le T_{n+1},\ \forall n\ge 0;$
\item $T_n<\infty$ implies $T_n<T_{n+1} \ \forall n\ge 0.$
\end{itemize}
The sequence $(T_n,\zeta_n)_{n\ge 0}$ is called a marked point process (MPP). Moreover, we assume the marked point process is non-explosive, i.e., $T_n\to\infty,\ \mathbb P-a.s.$

Define a random discrete measure $p$ on $((0,+\infty) \times E, \mathscr{B}((0,+\infty) \otimes \mathscr{E})$ associated with each MPP:
\begin{equation}
\label{eq p}
    p(\omega, D)=\sum_{n \geq 1} \mathbf{1}_{\left(T_n(\omega), \zeta_n(\omega)\right) \in D} .
\end{equation} 
For each $\tilde C \in \mathscr{E}$, define the counting process $N_t(\tilde C)=p((0, t] \times \tilde C)$ and denote $N_t=N_t(E)$. Obviously, both are right continuous increasing process starting from zero.  Define for $t \geq 0$
$$
\mathscr{G}_t^0=\sigma\left(N_s(\tilde C): s \in[0, t], \tilde C \in \mathscr{E}\right)
$$
and $\mathscr{G}_t=\sigma\left(\mathscr{G}_t^0, \mathscr{N}\right)$, where $\mathscr{N}$ is the family of $\mathbb{P}$-null sets of $\mathscr{F}$. 
Note by $\mathbb{G}=\left(\mathscr{G}_t\right)_{t \geq 0}$ the  completed filtration generated by the MPP, which is right continuous and satisfies the usual hypotheses.  Given a standard Brownian motion $W\in \mathbb R^d$, independent with the MPP, let $\mathbb F=(\mathcal F_t)$ be the completed filtration generated by the MPP and $W$, which satisfies the usual conditions as well.

Each marked point process has a unique compensator $v$, a predictable random measure such that
$$
\mathbb{E}\left[\int_0^{+\infty} \int_E C_t(e) p(d t d e)\right]=\mathbb{E}\left[\int_0^{+\infty} \int_E C_t(e) v(d t d e)\right]
$$
for all $C$ which is non-negative and $\mathscr{P}^{\mathscr{G}} \otimes \mathscr{E}$-measurable, where $\mathscr{P}^{\mathscr{G}}$ is the $\sigma$-algebra generated by $\mathscr{G}$-predictable processes. Moreover, in this paper we always assume that there exists a function $\phi$ on $\Omega \times[0,+\infty) \times \mathscr{E}$ such that $\nu(\omega, d t d e)=\phi_t(\omega, d e) d A_t(\omega)$, where $A$ is the dual predictable projection of $N$. In other words, $A$ is the unique right continuous increasing process with $A_0=0$ such that, for any non-negative predictable process $D$, it holds that,
$$
\mathbb E\left[\int_{0}^\infty D_t dN_t\right]=E\left[\int_{0}^\infty D_t dA_t\right].
$$
In addition, the kernel $\phi$ satisfies,
\begin{itemize}
\item for each $(w,t)\in \Omega\times [0,\infty)$, $\phi_t(w,\cdot)$ is a probability measure on $(E,\mathcal E)$;
\item for each $\tilde C\in \mathcal E$, $\phi_t(\tilde C)$ is predictable. 
\end{itemize}

\begin{example}[\cite{schonbucher2003credit}]
\label{ex1}
We list some marked point processes and their compensator measures in this example, which can be found in \cite[$\S 4.7$]{schonbucher2003credit}.
\begin{itemize}
\item Poisson Process $N_t$ with intensity $\lambda_t$ is a special marked point process with $E=\{1\}$. The compensator measure is $\nu(d e, d t)=\delta_{\zeta=1}(d e) \lambda_t d t$, and $d A_t=\lambda_t d t, \phi_t( de)=\delta_{\zeta=1}(de)$.
\item For a marked inhomogeneous Poisson Process with  marker $\zeta$ follows $N(0,1)$, the  compensator measure is $\nu(d e, d t)=\frac{1}{\sqrt{2 \pi}} e^{-1 / 2 e^2} \lambda_t d e d t$ and $d A_t=\lambda_t d t, \phi_t(d e)=\frac{1}{\sqrt{2 \pi}} e^{-1 / 2 e^2} d e$.
\item Another frequently used marked Poisson point process in financial modeling is that the marker $\zeta$ is the value of another stochastic process $S(t)$ at the time of the jump while $S(t)$ follows a geometric Brownian motion with $S(t)$ is observable at time $t$ :
$$
\frac{d S}{S}=\alpha d t+\sigma d W.
$$
The compensator measure reads,
$$
v(d e, d t)=\delta_{\zeta=S(t-)}(d e) \lambda_t d t ,
$$
which has no substaintial difference from the usual Poisson process due to the fact that the value of the marker is predictable. It follows that,
$
d A_t=\lambda_t d t, \phi_t(d e)=\delta_{\zeta=S(t-)}(d e).
$
\end{itemize}
\end{example}

Fix a terminal time $T>0$, we can define the integral
$$
\int_0^T \int_E C_t(e) q(d t d e)=\int_0^T \int_E C_t(e) p(d t d e)-\int_0^T \int_E C_t(e) \phi_t(d e) d A_t,
$$
under the condition
$$
\mathbb{E}\left[\int_0^T \int_E\left|C_t(e)\right| \phi_t(d e) d A_t\right]<\infty .
$$
Indeed, the process $\int_0^{\cdot} \int_E C_t(e) q(d t d e)$ is a martingale. Note that  $\int_a^b$ denotes an integral on $(a, b]$ if $b<\infty$, or on $(a, b)$ if $b=\infty$.

For $\beta>0$, we introduce the following spaces.
\begin{itemize}
    \item $L^{r, \beta}(A)$ is the space of all $\mathbb{F}$-progressive processes $X$ such that
$$
\|X\|_{L^{r, \beta}(A)}^r=\mathbb{E}\left[\int_0^T e^{\beta A_s}\left|X_s\right|^r d A_s\right]<\infty .
$$
    \item $L^{r, \beta}(p)$ is the space of all $\mathbb{F}$-predictable processes $U$ such that
$$
\|U\|_{L^{r, \beta}(p)}^r=\mathbb{E}\left[\int_0^T \int_E e^{\beta A_s}\left|U_s(e)\right|^r \phi_s(d e) d A_s\right]<\infty.
$$
    \item $L^{r, \beta}\left(W\right)$ is the space of $\mathbb{F}$-progressive processes $Z$ in $\mathbb{R}^d$ such that
$$
\|Z\|_{L^{r, \beta}(W)}^r=\mathbb{E}\left[\int_0^T e^{\beta A_s}\left|Z_s\right|^r d s\right]<\infty.
$$
    \item $S^2_*$ is the space of all $\mathbb{F}$-progressive processes $Y$ such that
$$
\|Y\|_{S^2_*}^2 = \sup_{0 \leq t \leq T} \mathbb{E} \left[  Y_{t}^2 \right] < \infty.
$$
    \item $\mathscr{A}_D$ is the space of all càdlàg non-decreasing deterministic processes $K$ starting from the origin, i.e. $K_0 = 0$.
\end{itemize}

\begin{remark}
As in \cite{Confortola2013}, we say that $X, X^{\prime} \in L^{r, \beta}(A)$ (respectively, $U, U^{\prime} \in L^{r, \beta}(p)$) are equivalent if they coincide almost everywhere with respect to the measure $d A_t(\omega) \mathbb{P}(d \omega)$ (respectively, the measure $\phi_t(\omega, d y) d A_t(\omega) \mathbb{P}(d \omega)$)  and this happens if and only if $\|X-X'\|_{L^{r, \beta}(A)}=0$ (respectively, $\|U-U'\|_{L^{r, \beta}(p)}=0$). With a little abuse of notation, we  still denote $L^{r, \beta}(A)$ (respectively, $L^{r, \beta}(p)$) the corresponding set of equivalence classes, endowed with the norm $\|\cdot\|_{L^{r, \beta}(A)}$ (respectively, $\|\cdot\|_{L^{r, \beta}(p)}$). In addition, both $L^{r, \beta}(A)$ and $L^{r, \beta}(p)$ are Banach spaces. 

\end{remark}

\subsection{Results on BSDEs driven by a marked point process}

The BSDE driven by a marked point process under consideration without reflection constraint is formulated as follows:
\begin{equation}
\label{eq standard}
 Y_t=\xi+\int_t^T f\left(s, Y_s, U_s\right) d A_s+\int_t^T g\left(s, Y_s, Z_s\right) d s-\int_t^T \int_E U_s(e) q(d s d e)-\int_t^T Z_s d W_s, \quad \forall t \in[0, T] \text { a.s., }
\end{equation}
which is a special case of the reflected BSDE mentioned in \cite{foresta2021optimal}.
The solution of (\ref{eq standard}) is a triple $(Y, U, Z)$ that lies in $L^{2, \beta}(A) \cap L^{2, \beta}(W) \times L^{2, \beta}(p) \times$ $L^{2, \beta}(W)$ for some $\beta>0$, with $Y$ c$\grave{a}$dl$\grave{a}$g. 

Hereafter, we are ready to state the general assumptions that will be used throughout the paper.\\[5pt]

\noindent\textbf{Assumption $\left(\mathbf{I}\right)$} The process A is continuous.\\[5pt]

The first assumption is on the dual predictable projection $A$ of the counting process $N$ relative to $p$. We would like to emphasize that for $A_t$, we do not require absolute continuity with respect to the Lebesgue measure. A celebrated example of interest in the theory of credit risk is given in \cite[Example 2.1]{janson2011absolutely}. For the readers' convenience, we restate this example below.
\begin{example}[\cite{janson2011absolutely}]
\label{ex2}
A special point process $1_{\{t\ge R\}}$, in which $R$ is a stopping time defined later, is considered in this example. 

Let $B$ be a standard one dimensional Brownian motion with natural filtration $\mathbb{F}$ and with a local time at zero $L=\left(L_t\right)_{t \geq 0}$. Define the change of time
$$
\tau_t=\inf \left\{s>0: L_s>t\right\} .
$$
Then $\left(\tau_t\right)_{t \geq 0}$ is a family of $\mathbb{F}$ stopping times. Form the time changed filtration $\mathbb{G}$ given by $\mathcal{G}_t=\mathcal{F}_{\tau_t}$ for $t \geq 0$. Enlarge $\mathbb G$ to $\mathbb H$ in order to support an independent Poisson process $N$ whose intensity $\lambda=1$. Then the family $\left(L_t\right)_{t \geq 0}$ are stopping times for $\mathbb{H}$, and
$N_{L_t}-L_t$ is a local martingale for the filtration $\tilde{\mathbb{H}}$
where $\left(\tilde{\mathcal{H}}_t\right)=\left(\mathcal{H}_{L_t}\right)_{t \geq 0}$. 

Since $L$ is Brownian local time at zero, it has paths which are singular with respect to Lebesgue measure, a.s. However by Tanaka's formula,
\begin{equation}
\label{Tanaka}
\left|B_t\right|=\int_0^t \operatorname{sign}\left(B_s\right) d B_s+L_t,
\end{equation}
which implies,
$\mathbb E\left(L_t\right)=\mathbb E\left(\left|B_t\right|\right)=\sqrt{\frac{2}{\pi}} \sqrt{t}$.

Next denote
$$
R=\inf \left\{s>0: N_{L_s} \geq 1\right\},
$$
which is a stopping time for the filtration $\tilde{\mathbb{H}}$. Then $1_{\{t \geq R\}}-L_{t \wedge R}$ is a martingale for the filtration $\tilde{\mathbb{H}}$. Hence, for the filtration $\tilde{\mathbb{H}}$, the dual predictive projection for $1_{\{t \geq R\}}$ is $L_{t \wedge R}$, which is continuous but singular with respect to the Lebesgue measure. $1_{\{t \geq R\}}$ can be viewd as a special marked point process with marker $E=\{1\}$, whose compensator measure is $v(de,dt)=\delta_{\zeta=1}(de)dL_{t\wedge R}$. Without much difficulty, we are able to construct similar processes in Example \ref{ex1} by replacing Poisson process $N_t$ with $1_{\{t\ge R\}}$.
\end{example}

\noindent\textbf{Assumption $\left(\mathbf{I I}\right)$}\\
(i) The final condition $\xi: \Omega \rightarrow \mathbb{R}$ is $\mathscr{F}_T$-measurable and
$$
\mathbb{E}\left[e^{\beta A_T} \xi^2\right]<\infty;
$$
(ii) For every $\omega \in \Omega, t \in[0, T], r \in \mathbb{R}$ a mapping
$$
f(\omega, t, r, \cdot): L^2\left(E, \mathscr{E}, \phi_t(\omega, d e)\right) \rightarrow \mathbb{R}
$$
is given and satisfies the following:

(a) For every $U \in L^{2, \beta}(p)$ the mapping
$$
(\omega, t, r) \mapsto f\left(\omega, t, r, U_t(\omega, \cdot)\right)
$$

\quad \quad is $\mathscr{P} \otimes \mathscr{B}(\mathbb{R})$-measurable, where $\mathscr{P}$ denotes the progressive $\sigma$-algebra with respect to $\mathbb F$.

(b) There exist $L_f \geq 0, L_p \geq 0$ such that for every $\omega \in \Omega, t \in[0, T], y, y^{\prime} \in \mathbb{R}$, $u, u^{\prime} \in L^2\left(E, \mathscr{E}, \phi_t(\omega, d e)\right)$ we have
$$
\begin{aligned}
& \left|f(\omega, t, y, u(\cdot))-f\left(\omega, t, y^{\prime}, u^{\prime}(\cdot)\right)\right| \leq L_f\left|y-y^{\prime}\right|+L_p\left(\int_E\left|u(e)-u^{\prime}(e)\right|^2 \phi_t(\omega, d e)\right)^{1 / 2}.
\end{aligned}
$$

(c) We have
$$
\mathbb{E}\left[\int_0^T e^{\beta A_s}|f(s, 0,0)|^2 d A_s\right]<\infty.
$$
(iii) The mapping $g: \Omega \times[0, T] \times \mathbb{R} \times \mathbb{R}^d \rightarrow \mathbb{R}$ is given

(a) $g$ is $\mathscr{P} \times \mathscr{B}(\mathbb{R}) \times \mathscr{B}\left(\mathbb{R}^d\right)$ measurable.

(b) There exist $L_g \geq 0, L_w \geq 0$ such that for every $\omega \in \Omega, t \in[0, T], y, y^{\prime} \in \mathbb{R}$, $z, z^{\prime} \in \mathbb{R}^d$
$$
\left|g(\omega, t, y, z)-g\left(\omega, t, y^{\prime}, z^{\prime}\right)\right| \leq L_g\left|y-y^{\prime}\right|+L_w\left|z-z^{\prime}\right|.
$$

(c) We have
$$
\mathbb{E}\left[\int_0^T e^{\beta A_s}|g(s, 0,0)|^2 d s\right]<\infty.
$$

The well-posedness of BSDE (\ref{eq standard}) inherits from \cite[Proposition 3.3]{foresta2021optimal}. 

\begin{prop}[\cite{foresta2021optimal}]
\label{prop2}
    Let assumptions $(\mathbf{I})$ and $\left(\mathbf{II}\right)$ hold for some $\beta > L_p^2 + 2L_f$, then BSDE (\ref{eq standard}) admits a unique solution $(Y, U, Z) \in L^{2, \beta}(A) \cap$ $L^{2, \beta}(W) \times L^{2, \beta}(p) \times L^{2, \beta}(W)$, with $Y$ c$\grave{a}$dl$\grave{a}$g.
\end{prop}
 The martingale representation theorem for c\`adl\`ag square integrable $\mathbb{F}$-martingale is essential when constructing the solution. We omit the proof for brevity.

\section{Formulation of the problem} 
\label{sec F}
In this paper, we consider the BSDE (\ref{eq standard}) with a mean reflected condition which is as follows: for some $\beta>0$,
\begin{equation}
\label{eq1}
\left\{\begin{array}{l}
Y_t =\xi+\int_t^T f\left(s, Y_s, U_s\right) d A_s+\int_t^T g\left(s, Y_s, Z_s\right) d s\\
\quad \quad-\int_t^T \int_E U_s(e) q(d s d e)-\int_t^T Z_s d W_s + \left(K_T-K_t\right) , \quad \forall t \in[0, T] \text { a.s. };\\
\mathbb{E}\left[\ell\left(t, Y_t\right)\right] \geq 0, \quad \forall t \in[0, T] ,\\
(Y, U, Z, K) \in L^{2, \beta}(A) \cap L^{2, \beta}(W) \times L^{2, \beta}(p) \times L^{2, \beta}(W) \times \mathcal{A}_D.
\end{array}\right.
\end{equation}

Consistently, the flatness is defined by the following Skorokhod condition:
$$
    \int_0^T \mathbb{E}\left[\ell\left(t, Y_{t^{-}}\right)\right] d K_t=0.
$$
Here, some requirements of the running loss function $\ell$ is needed.\\
\textbf{Assumption $\left(\mathbf{I I I}\right)$} $\ell:\Omega \times [0, T] \times \mathbb{R} \rightarrow \mathbb{R}$ satisfies the following properties:

1. $(t, y) \rightarrow \ell(t, y)$ is uniformly continuous, uniformly in $\omega$,

2. $\forall t \in[0, T], y \rightarrow \ell(t, y)$ is strictly increasing,

3. $\forall t \in[0, T], \lim_{x\to\infty} \mathbb{E}\left[ \ell(t, x)\right]>0$,

4. $\forall t \in[0, T], \forall y \in \mathbb{R},|\ell(t, y)| \leq C(1+|y|)$ for some constant $C \geq 0$.\\
\textbf{Assumption $\left(\mathbf{IV}\right)$} There exist two constants $\overline\kappa>\underline\kappa>0$ such that for each $t\in[0, T]$ and $y_1, y_2 \in \mathbb{R}$,
$$
\underline{\kappa}\left|y_1-y_2\right| \leq\left|\ell\left(t, y_1\right)-\ell\left(t, y_2\right)\right| \leq \overline\kappa\left|y_1-y_2\right| .
$$

In order to study mean reflected BSDEs, we construct the following map $L_t: L^{2, \beta}(A) \cap L^{2, \beta}(W) \rightarrow \mathbb{R}$ for each $t \in[0, T]:$
$$
L_t(\eta)=\inf \{x \geq 0: \mathbb{E}[\ell(t, x+\eta)] \geq 0\}, \quad \forall \eta \in L^{2, \beta}(A) \cap L^{2, \beta}(W).
$$
When assumption $\left(\mathbf{I I I }\right)$ is satisfied, the operator $X \mapsto L_t(X)$ is well-defined, similar with \cite{briand2018bsdes}.
\begin{remark}
\label{remark1}
Moreover, if assumption $\left(\mathbf{IV}\right)$ is also fulfilled, then for each $t \in[0, T]$, $\kappa:=\overline\kappa/\underline\kappa>1$,
\begin{equation}
\label{eqL}
\left|L_t\left(\eta^1\right)-L_t\left(\eta^2\right)\right| \leq \kappa \mathbb{E}\left[\left|\eta^1-\eta^2\right|\right], \quad \forall \eta^1, \eta^2 \in L^{2, \beta}(A) \cap L^{2, \beta}(W) .
\end{equation}
\end{remark}

\section{Fixed generator case}
\label{sec B}
In this section, we first consider the simple case that the generators $g$ and $f$ do not depend on $(Y , Z, U)$ with mean reflection, i.e, BSDE with mean reflection:
\begin{equation}
\label{eq2}
\left\{\begin{array}{l}
Y_t =\xi+\int_t^T f_s d A_s+\int_t^T g_s d s-\int_t^T \int_E U_s(e) q(d s d e)-\int_t^T Z_s d W_s \\
\quad \quad+ \left(K_T-K_t\right) , \quad \forall t \in[0, T] \text { a.s. };\\
\mathbb{E}\left[\ell\left(t, Y_t\right)\right] \geq 0, \quad \forall t \in[0, T] .
\end{array}\right.
\end{equation}

The following simplified assumption is needed.

\textbf{Assumption $\left(\mathbf{I I}^{\prime}\right) $} $f$ and $g$ are $\mathbb{F}$-progressive processes such that
$$
\mathbb{E}\left[\int_0^T e^{\beta A_s}\left|f_s\right|^2 d A_s+\int_0^T e^{\beta A_s}\left|g_s\right|^2 d s\right]<\infty .
$$

The main result of this section reads as follows.
\begin{thm}
\label{thm1}
Assume that assumptions $(\mathbf{I})$, $\left(\mathbf{II}\right)$-(i) and $\left(\mathbf{II}^{\prime}\right)$ hold for some $\beta > 0$. Meanwhile, $\left(\mathbf{I I I}\right)$ and $\left(\mathbf{IV}\right)$ hold, then the BSDE (\ref{eq2}) with mean reflection admits a unique deterministic flat  solution $(Y, U, Z, K) \in L^{2, \beta}(A) \cap$ $L^{2, \beta}(W) \times L^{2, \beta}(p) \times L^{2, \beta}(W) \times \mathcal{A}_D$. 
\end{thm}

The following a priori estimate on $Y$, similar to that in Foresta \cite{foresta2021optimal}, is essential in the proof of Theorem \ref{thm1}.

\begin{lemma}[A priori estimate on $Y$]
\label{A priori estimate}
Assume $\left(\mathbf{II}\right)$-(i) and $\left(\mathbf{II}^{\prime}\right)$ hold, then,
$$
\mathbb{E}\left[\sup _{t \in[0, T]} e^{\beta A_t} Y_t^2\right]<\infty .
$$
\end{lemma}

With the help of this a priori estimate, we are going to prove Theorem \ref{thm1}.
\begin{proof}
\textbf{Step 1: Existence.}
Consider the following BSDE:
\begin{equation}
\label{eq small fixed}
y_t=\xi+\int_t^T f_s d A_s+\int_t^T g_s d s-\int_t^T \int_E u_s(e) q(d s d e)-\int_t^T z_s d W_s.
\end{equation}
In the spirit of the proof of \cite[Proposition 3.3]{foresta2021optimal}, we know BSDE (\ref{eq small fixed}) has a unique solution $(y, u, z) \in L^{2, \beta}(A) \cap$ $L^{2, \beta}(W) \times L^{2, \beta}(p) \times L^{2, \beta}(W)$.

Thus, inspired by \cite{briand2018bsdes,Briand_2020}, we can define:
$$
k_t=\sup_{0\le s\le T}L_s(y_s)-\sup _{t\le s\le T}L_s(y_s).
$$

We first show that $s\to L_s(y_s)$ is right continuous. Obviously $y_s$ is a right continuous process. Suppose there are two constants $x,y$ satisfying $x < L_t(y_t) < y$, and there exists $\epsilon>0$ such that $0 < s-t < \epsilon$, then, due to assumption $(\mathbf{III})$,
$$
\begin{aligned}
    \lim _{s \downarrow t} \mathbb{E} [l(s,x+ y_s)] &= \mathbb{E} [l(t,x+ y_t)] < \mathbb{E} [l(t,L_t(y_t)+ y_t)] =0 < \mathbb{E} [l(t,y+ y_t)] = \lim _{s \downarrow t} \mathbb{E} [l(s,y+ y_s)].
\end{aligned}
$$
Hence, for small enough $\epsilon$, $\mathbb{E} [l(s,x+ X_s)]< 0 < \mathbb{E} [l(t,y+ X_s)]$ implies that $x< L_s(X_s) <y$. Thus $L_s(y_s)$ is right continuous with respect to $s$. So we know that $k_t$ is a non-decreasing deterministic right continuous process with $k_0 = 0$. In the same manner, we can deduce that $k_t$ is c$\grave{a}$dl$\grave{a}$g.

Obviously, 
 $$
 \mathbb{E}\left[ \ell\left(t, y_t+k_T-k_t\right)\right]=
 \mathbb{E}\left[ \ell\left(t, y_t+\sup _{t \leqslant s \leqslant T} L_s(y_s ) \right)\right] \geq 0.
 $$
 Thus, let $Y=(y+k_T-k_t)$, $U = u$, $Z = z$, $K = k$,  and $(Y, U, Z, K)\in L^{2, \beta}(A) \cap$ $L^{2, \beta}(W) \times L^{2, \beta}(p) \times L^{2, \beta}(W) \times \mathcal{A}_D$ is a deterministic solution to the BSDE with mean reflection (\ref{eq2}). 

\textbf{Step 2: Flat and Uniqueness.} The idea borrows from the proof of  \cite[Proposition 7]{briand2018bsdes}.
  
We first verify that the solution is flat.

Observe that with the help of the right continuity of $s\to L_s(y_s)$ and the definition of $K$, 
$\mathbb{E}\left[\ell\left(t, y_{t^{-}} + \sup _{t\leqslant s \leqslant T}L_{s^-}(y_{s^-} ) \right)\right] = \mathbb{E}\left[\ell\left(t, y_{t^{-}} + L_{t^{-}}(y_{t^{-}} ) \right)\right]$, $dK-a.e$ and $L_{t^{-}}(y_{t^{-}})>0$, $dK-a.s.$.
So we have, 
$$
\begin{aligned}
    &\int_0^T \mathbb{E}\left[\ell\left(t, Y_{t^{-}}\right)\right] d K_t\\
    &= \int_0^T \mathbb{E}\left[\ell\left(t, y_{t^{-}} + \sup _{t\leqslant s \leqslant T}L_{s^-}(y_{s^-} ) \right)\right] d K_t\\
    &= \int_0^T \mathbb{E}\left[\ell\left(t, y_{t^{-}} + L_{t^{-}}(y_{t^{-}}) \right)\right] d K_t\\
    &= \int_0^T \mathbb{E}\left[\ell\left(t, y_{t^{-}} + L_{t^{-}}(y_{t^{-}}) \right)\right] \mathbf{1}_{\{L_{t^{-}}(y_{t^{-}})>0\}} d K_t\\
    &= \int_0^T \mathbb{E}\left[\ell\left(t^-, y_{t^{-}} + L_{t^{-}}(y_{t^{-}}) \right)\right] \mathbf{1}_{\{L_{t^{-}}(y_{t^{-}})>0\}} d K_t\\
    &=0.
\end{aligned}
$$
The last equality follows from the continuity of $\ell(\cdot,x)$. Thus $(Y, U, K)$ is a flat solution.

We are at the position to prove the uniqueness of the deterministic flat solution of mean reflected BSDE (\ref{eq2}). We prove by contradiction.

Suppose $\left(Y^1, U^1, K^1\right)$ and $\left(Y^2, U^2, K^2\right)$ are two different deterministic flat solutions to (\ref{eq2}). Thus both  $(Y^1_t-K_T^1+K_t^1,U_t^1)$ and  $(Y^2_t-K_T^2+K_t^2,U_t^2)$ are solutions to the standard BSDE (\ref{eq small fixed}). It follows from the uniqueness of the standard BSDE (\ref{eq small fixed}) that  $Y_t^1-K_T^1+K_t^1= Y_t^2-K_T^2+K_t^2$ and $U_t^1=U_t^2$ for each $t\in[0,T]$,  where the equations stand for being in the same equivalent class. Thus, there exists $t_1<T$ such that either,
$$
K_T^1-K_{t_1}^1>K_T^2-K_{t_1}^2, 
$$
or,
$$
K_T^2-K_{t_1}^2>K_T^1-K_{t_1}^1.
$$
Without loss of generality, we suppose the former case. Define $t_2$ as the first time after $t_1$ such that
$$
K_T^1-K_{t_2^-}^1=K_T^2-K_{t_2^-}^2.
$$
Note that for each $t\in(t_1,t_2]$, 
$
K_T^1-K_{t^-}^1\ge K_T^2-K_{t^-}^2. 
$
Two different scenarios may happen.
\begin{itemize}
\item Scenario 1: $t_2\le T$: 
\end{itemize}
In this case, $Y_{t^-}^1>Y_{t^-}^2$ for each $t\in(t_1,t_2)$. In view of the fact that $\ell(t,x)$ is strictly increasing in $x$,
$$
\mathbb{E}\left[\ell\left(t, Y_{t^{-}}^1\right)\right] > \mathbb{E}\left[\ell\left(t, Y_{t^{-}}^2\right)\right] \geq 0, \quad t_1 < t < t_2 .
$$
 However, $\left(Y^1, U^1, K^1\right)$ is a flat solution and via Skorohod condition,
$
dK_t^1=0,
$
for each $t\in(t_1,t_2)$. 
Thus, $K_{t_2^-}^1 = K_{t_1}^1$. We deduce that
$$
K_T^1-K_{t_2^-}^1=K_T^1-K_{t_1}^1>K_T^2-K_{t_1}^2 \geq K_T^2-K_{t_2^-}^2,
$$
which contradicts the definition of $t_2$.

\begin{itemize}
\item Scenario 2: $t_2=\infty$: 
\end{itemize}
It turns out that in this case, $Y_{t^-}^1>Y_{t^-}^2$ for each $t\in(t_1,T]$. Similarly, by means of the fact that $\ell(t,x)$ is strictly increasing in $x$,
$$
\mathbb{E}\left[\ell\left(t, Y_{t^{-}}^1\right)\right] > \mathbb{E}\left[\ell\left(t, Y_{t^{-}}^2\right)\right] \geq 0, \quad t_1 < t \le T.
$$
Then, via Skorohod condition again,
$
dK_t^1=0,
$
for each $t\in(t_1,T]$. 
Thus, $K_{T}^1 = K_{t_1}^1$. We deduce that
$$
0=K_T^1-K_{T}^1=K_T^1-K_{t_1}^1>K_T^2-K_{t_1}^2\ge 0,
$$
which also leads to a contradiction.

The two scenarios above together imply the uniqueness of the deterministic flat solution of mean reflected BSDE ((\ref{eq2}).

\end{proof}

\section{General generator case}
\label{sec ger}
In this section, the general generator case is taken into consideration: 
\begin{equation}
\label{eq3}
\left\{\begin{array}{l}
Y_t =\xi+\int_t^T f\left(s, Y_s, U_s\right) d A_s+\int_t^T g\left(s, Y_s, Z_s\right) d s \\
\quad \quad-\int_t^T \int_E U_s(e) q(d s d e)-\int_t^T Z_s d W_s+ \left(K_T-K_t\right) , \quad \forall t \in[0, T] \text { a.s. };\\
\mathbb{E}\left[\ell\left(t, Y_t\right)\right] \geq 0, \quad \forall t \in[0, T] .
\end{array}\right.
\end{equation}
In order to give the well-posedness of BSDE (\ref{eq3}),we add one assumption as follow:\\
\textbf{Assumption $\left(\mathbf{V}\right)$}: 
$$
\mathbb{E}\left[ e^{\beta A_T} \right] < \infty.
$$

\begin{remark}
\label{remark bdd}
If assumption $\left(\mathbf{V}\right)$ holds, we can deduce that:
$$
\mathbb{E}\left[ \int_{0}^{T} e^{\beta (A_s+ s)} (d A_s+ d s)\right] < \infty.
$$
\end{remark}

\begin{remark}
Assumption $\left(\mathbf{V}\right)$ is  easily satisfied by many typical marked point processes. For processes in Example \ref{ex1} with absolutely continuous dual predictive projection, Assumption $\left(\mathbf{V}\right)$ is equivalent to $\mathbb{E}\left[ e^{\beta\int_0^T\lambda_tdt} \right] < \infty$, which is fulfilled by properly chosen intensity $\lambda_t$. The process in Example {\ref{ex2}}, $1_{\{t\ge R\}}$, satisfies Assumption $\left(\mathbf{V}\right)$ automatically, since $\mathbb{E}\left[ e^{\beta L_T} \right] < \infty$ holds for any $\beta >0 $, deduced from Tanaka's formula (\ref{Tanaka}). 
\end{remark}

The main result of this section reads as follows.
\begin{thm}
\label{thm2}
    Let assumptions $(\mathbf{I})$, $\left(\mathbf{II}\right)$ and $\left(\mathbf{V}\right)$ hold for some $\beta$ satisfying  $\beta > 256 \kappa^4[3(L_f + L_g) + 2(L_p^2 + L_w^2)]$, assumptions $\left(\mathbf{I I I}\right)$ and $\left(\mathbf{IV}\right)$ hold, then the BSDE (\ref{eq3}) with mean reflection admits a unique deterministic flat  solution $(Y, U, Z, K) \in L^{2, \beta}(A) \cap$ $L^{2, \beta}(W) \times L^{2, \beta}(p) \times L^{2, \beta}(W) \times \mathcal{A}_D$. 
\end{thm}

In order to prove Theorem \ref{thm2}, we introduce a representation result which plays a key role in establishing the existence and uniqueness result.
\begin{lemma}
\label{lemma rep}
    Assume assumptions $(\mathbf{I})$, $\left(\mathbf{II}\right)$ and $\left(\mathbf{I I I}\right)$ hold for some $\beta >0$. Suppose $(Y, U, Z, K) \in L^{2, \beta}(A) \cap$ $L^{2, \beta}(W) \times L^{2, \beta}(p) \times L^{2, \beta}(W) \times \mathcal{A}_D$ is a deterministic flat solution to the BSDE with mean reflection (\ref{eq3}). Then, for each $t \in[0, T]$
$$
\left(Y_t, U_t, Z_t, K_t \right)=\left(y_t+\sup _{t \leq s \leq T} L_s\left(y_s\right), u_t, z_t, \sup _{0 \leq s \leq T} L_s\left(y_s\right)-\sup _{t \leq s \leq T} L_s\left(y_s\right)\right),
$$
where $(y, u, z) \in L^{2, \beta}(A) \cap$ $L^{2, \beta}(W) \times L^{2, \beta}(p) \times L^{2, \beta}(W) $ is the solution to the following BSDE with the driver $f\left(s, Y_s, U_s\right), g\left(s, Y_s, Z_s\right)$ on the time horizon $[0, T]$ , and $Y \in L^{2, \beta}(A) \cap$ $L^{2, \beta}(W)$ is fixed by the solution of (\ref{eq3}):
\begin{equation}
{\label{eq samll}}
y_t=\xi+\int_t^T f\left(s, Y_s, U_s\right) d A_s+\int_t^T g\left(s, Y_s, Z_s\right) d s-\int_t^T \int_E u_s(e) q(d s d e)-\int_t^T z_s d W_s.
\end{equation}
\end{lemma}

\begin{proof}
     First show (\ref{eq samll}) has a unique solution. It is equal to show that $f\left(s, Y_s, U_s\right), g\left(s, Y_s, Z_s\right)$ satisfies assumption $\left(\mathbf{II}^{\prime}\right)$.
     
 By assumption $\left(\mathbf{II}\right)$ and the fact $(Y, U, Z) \in L^{2, \beta}(A) \cap$ $L^{2, \beta}(W) \times L^{2, \beta}(p) \times L^{2, \beta}(W)$:
$$
\begin{aligned}
   &\mathbb{E}\left[\int_0^{T} e^{\beta A_s} \left| f\left(s, Y_s, U_s\right) \right|^2 d A_s\right]+\mathbb{E}\left[\int_0^{T} e^{\beta A_s} \left|g\left(s, Y_s, Z_s\right) \right|^2 d s\right] \\
   &\leq \mathbb{E}\left[\int_0^{T} e^{\beta A_s} \left[ |f(s, 0, 0)| +|L_f Y_s| + L_p\left(\int_E U_s(e)\phi_s(d e)\right)^{\frac{1}{2}} \right]^2 d A_s\right]\\
   &\quad +\mathbb{E}\left[\int_0^{T} e^{\beta A_s} \left[ |g(s, 0, 0)| +|L_g Y_s| + |L_w Z_s| \right]^2 d s\right] \\
   &\le2(1+L_f^2+L_p^2+L_g^2+L_w^2)\left(  \mathbb{E}\left[\int_0^{T} e^{\beta A_s}|f(s, 0, 0)|^2 d A_s\right] + \mathbb{E}\left[\int_0^{T} e^{\beta A_s} Y_s^2 d A_s\right]  + \mathbb{E}\left[\int_0^{T}\int_E e^{\beta A_s} U_s(e)\phi_s(d e) d A_s\right]\right. \\
   &\quad \left. +  \mathbb{E}\left[\int_0^{T} e^{\beta A_s}|g(s, 0, 0)|^2 d s\right] + \mathbb{E}\left[\int_0^{T} e^{\beta A_s} Y_s^2 d s\right]  + \mathbb{E}\left[\int_0^{T}\int_E e^{\beta A_s} Z_s^2 d s\right]\right) < \infty.
\end{aligned}
$$
Thus, (\ref{eq samll}) has a unique solution $(y, u, z) \in L^{2, \beta}(A) \cap$ $L^{2, \beta}(W) \times L^{2, \beta}(p) \times L^{2, \beta}(W)$. 

Define 
$$
k_t=\sup_{0\le s\le T}L_s(y_s)-\sup _{t\le s\le T}L_s(y_s).
$$
With the help of the proof of Theorem \ref{thm1}, $(y_t+k_T-k_t,u_t,z_t,k_t)\in  L^{2, \beta}(A) \cap$ $L^{2, \beta}(W) \times L^{2, \beta}(p) \times L^{2, \beta}(W) \times \mathcal A_D$ is the unique deterministic flat solution of the following BSDE with mean reflection with driver $f\left(s, Y_s, U_s\right), g\left(s, Y_s, Z_s\right)$:
\begin{equation}{\label{tilde}}
\left\{\begin{array}{l}
\tilde{Y}_t =\xi+\int_t^T f\left(s, Y_s, U_s\right) d A_s+\int_t^T g\left(s, Y_s, Z_s\right) d s\\
\quad \quad-\int_t^T \int_E \tilde{U}_s(e) q(d s d e)-\int_t^T \tilde{Z}_s d W_s + \left(\tilde{K}_T-\tilde{K}_t\right) , \quad \forall t \in[0, T] \text { a.s. };\\
\mathbb{E}\left[\ell\left(t, \tilde{Y}_t\right)\right] \geq 0, \quad \forall t \in[0, T] .
\end{array}\right.
\end{equation}

Notice that $(Y,U,Z,K)$ is also a deterministic flat solution to (\ref{tilde}). By uniqueness, $(Y_t,U_t,Z_t,K_t)=(y_t+k_T-k_t,u_t,z_t,k_t)$.

Therefore, 
$$
\left(Y_t,U_t,Z_t,K_t \right)=\left(y_t+\sup _{t \leq s \leq T} L_s\left(y_s\right), u_t, z_t, \sup _{0 \leq s \leq T} L_s\left(y_s\right)-\sup _{t \leq s \leq T} L_s\left(y_s\right)\right).
$$
\end{proof}

Replace $(Y,U,Z)$ by a general $(P,Q,R)\in L^{2, \beta}(A) \cap$ $L^{2, \beta}(W) \times L^{2, \beta}(p) \times L^{2, \beta}(W)$, we have the following conclusion. 
\begin{lemma}
\label{lemma PQR}
Assume assumptions $(\mathbf{I})$, $\left(\mathbf{II}\right)$ and $\left(\mathbf{I I I}\right)$ hold and $(P,Q,R) \in L^{2, \beta}(A) \cap$ $L^{2, \beta}(W) \times L^{2, \beta}(p) \times L^{2, \beta}(W)$. Then, the BSDE (\ref{eq PQR}) with mean reflection admits a unique flat solution $(Y, U, Z, K) \in L^{2, \beta}(A) \cap$ $L^{2, \beta}(W) \times L^{2, \beta}(p) \times L^{2, \beta}(W) \times \mathcal{A}_D$.

\begin{equation}
\label{eq PQR}
\left\{\begin{array}{l}
Y_t =\xi+\int_t^T f\left(s, P_s, Q_s\right) d A_s+\int_t^T g\left(s, P_s, R_s\right) d s\\
\quad \quad-\int_t^T \int_E U_s(e) q(d s d e)-\int_t^T Z_s d W_s + \left(K_T-K_t\right) , \quad \forall t \in[0, T] \text { a.s. };\\
\mathbb{E}\left[\ell\left(t, Y_t\right)\right] \geq 0, \quad \forall t \in[0, T] .
\end{array}\right.
\end{equation}
In addition, with the help of Lemma \ref{A priori estimate}, $Y\in S_*^2$.
\end{lemma}
\begin{proof}
 Similar with Lemma \ref{lemma rep}, $(y, u, z) \in L^{2, \beta}(A) \cap$ $L^{2, \beta}(W) \times L^{2, \beta}(p) \times L^{2, \beta}(W) $ is the unique solution of 
 $$
y_t=\xi+\int_t^T f\left(s, P_s, Q_s\right) d A_s+\int_t^T g\left(s, P_s, R_s\right) d s-\int_t^T \int_E u_s(e) q(d s d e)-\int_t^T z_s d W_s. 
 $$
 Define $$ k_t = \sup _{0 \leqslant s \leqslant T} L_s(y_{s})-\sup _{t \leqslant s \leqslant T} L_s(y_s ).$$

 It follows that, $(Y_t, U_t, Z_t, K_t) = (y_t + (k_T - k_t), u_t, z_t, k_t)\in L^{2, \beta}(A) \cap$ $L^{2, \beta}(W) \times L^{2, \beta}(p) \times L^{2, \beta}(W)\times \mathcal A_D$ is a solution of:
 $$
 \tilde{Y}_t=\xi+\int_t^T f\left(s, P_s, Q_s\right) d A_s+\int_t^T g\left(s, P_s, R_s\right) d s-\int_t^T \int_E \tilde{U}_s(e) q(d s d e)-\int_t^T \tilde{Z}_s d W_s+ (\tilde{K}_T - \tilde{K}_t).
 $$
 Same as Lemma \ref{thm1}, we can check the mean reflection condition, the flatness and the uniqueness.
 So $(Y_t, U_t, Z_t, K_t) \in L^{2, \beta}(A) \cap$ $L^{2, \beta}(W) \times L^{2, \beta}(p) \times L^{2, \beta}(W)\times \mathcal A_D$ is the unique flat solution of (\ref{eq PQR}). The fact $Y\in S_*^2$ follows from Lemma \ref{A priori estimate}.
\end{proof}

Now we are going to prove Theorem \ref{thm2}. Inspired by \cite{hu2022general, Hibon2017quadratic}, we first prove the existence and uniqueness of the solution on a small time interval and then stitch the local solutions to build the global solution.

\begin{proof}[Proof of Theorem \ref{thm2}]

For notation simplicity, we denote $\mathbb{L^{\beta}} =L^{2, \beta}(A) \cap L^{2, \beta}(W) \times L^{2, \beta}(p) \times L^{2, \beta}(W)$. $\mathbb{L^{\beta}}(s,t)$ denotes the restriction of the space on $(s,t)$.
Define a solution map $\Gamma$ for $(P,Q,R) \in  \mathbb L^\beta$ by $\Gamma(P,Q,R) = (Y,U,Z)$ where $(Y,U,Z)$ is the first three components of the solution $(Y, U, Z, K)$ to (\ref{eq PQR}). Inheriting from Lemma \ref{lemma PQR},  $\Gamma\left( \mathbb L^\beta(T-h, T)\right) \subset  \mathbb L^\beta(T-h, T)$ for any $h \in(0, T]$.

Let $(P^i,Q^i,R^i) \in \mathbb L^\beta(T-h, T)$  for $i = 1,2$. It follows from Lemma \ref{lemma rep} that 
\begin{equation}
\label{eq 8}
\Gamma\left(U^i\right)_t:=y_t^i+\sup _{t \leq s \leq T} L_s\left(y_s^i\right), \quad \forall t \in[T-h, T],
\end{equation}
where $y^i$ is the solution to the BSDE (\ref{eq PQR}). Since $\mathbb L^{\beta}$ is complete, in view of Lemma \ref{lemma rep}, the existence and uniqueness of local solution in $[T-h,T]$ boils down to the strict contractivity of $\Gamma$  with respect to the norm $\|\cdot\|_{\mathbb L^\beta}$.

Denote $\bar{y} = y^1 - y^2$ and similarly denote $\bar{Y},\bar{U}, \bar{Z}$, $\bar{f}_s = f(s,P^1_s, Q^1_s) - f(s,P^2_s, Q^2_s)$ $\bar{g}_s = g(s,P^1_s, Q^1_s) - g(s,P^2_s, Q^2_s)$. Obviously, $(\bar{y}, \bar{U}, \bar{Z})$ satisfies
$$
\bar{y}=\int_t^T \bar{f}_s d A_s+\int_t^T \bar{g}_s d s-\int_t^T \int_E \bar{U}_s(e) q(d s d e)-\int_t^T \bar{Z}_s d W_s.
$$

For some $\beta>0$, applying  It$\hat o$'s formula to $e^{\beta (A_s+ s)} \bar{y}_s^2$ on $[T-h, T]$,
\begin{equation}
\label{eqito}
\begin{aligned}
e^{\beta (A_T +  T)} \bar{y}_T^2 =& e^{\beta (A_{T-h}+ (T-h))} \bar{y}_{T-h}^2 + \beta \int_{T-h}^{T} e^{\beta (A_s+ s)} \bar{y}_s^2 d A_s+\beta \int_0^{T} e^{\beta (A_s+ s)} \bar{y}_s^2 d s \\
& -2 \int_{T-h}^{T} e^{\beta (A_s+ s)} \bar{y}_s \bar{f}_s d A_s-2 \int_{T-h}^{T} e^{\beta (A_s+ s)} \bar{y}_s \bar{g}_s d s +2 \int_{T-h}^{T} e^{\beta (A_s+ s)} \bar{y}_s \bar{Z}_s d W_s\\
& +\int_{T-h}^{T} e^{\beta (A_s+ s)} \bar{Z}_s^2 d s +2 \int_{T-h}^{T} \int_E e^{\beta (A_s+ s)} \bar{y}_{s^{-}} \bar{U}_s(e) q(d s d e) +\int_{T-h}^{T} \int_E e^{\beta (A_s+ s)} \bar{U}_s^2(e) p(d s d e).
\end{aligned}
\end{equation}
Making use of the fact that
$$
\int_0^t \int_E U_s(e) p(d s d e)=\int_0^t \int_E U_s(e) \phi_s(d e) d A_s+\int_0^t \int_E U_s(e) q(d s d e),
$$ 
and taking expectation on both sides of (\ref{eqito}), we obtain:
$$
\begin{aligned}
& \beta \mathbb{E}  {\left[\int_{T-h}^{T} e^{\beta (A_s+ s)} \bar{y}_s^2 d A_s\right]+\beta \mathbb{E}\left[\int_{T-h}^{T} e^{\beta (A_s+ s)} \bar{y}_s^2 d s\right] } \\
& +\mathbb{E}\left[\int_{T-h}^{T} e^{\beta (A_s+ s)}\bar{Z}_s^2 d s\right]  +\mathbb{E}\left[\int_{T-h}^{T} \int_E e^{\beta(A_s+ s)} \bar{U}_s^2 \phi_s(d e) d A_s\right] \\
& \leq 2 \mathbb{E}\left[\int_{T-h}^{T} e^{\beta (A_s+ s)} \bar{y}_s\bar{f}_s d A_s\right]+2 \mathbb{E}\left[\int_{T-h}^{T} e^{\beta (A_s+ s)} \bar{y}_s\bar{g}_s d s\right].
\end{aligned}
$$

Then, for any $ t \in [0,T]$, applying  It$\hat o$'s formula to $e^{\beta (A_s+ s)} \bar{y}_s^2$ on $[t, T]$ and taking expectation we can observe that:
\begin{equation}
\label{eq ito}
\begin{aligned}
\mathbb E\left[e^{\beta (A_t+ t)} \bar{y}_{t}^2\right] \leq \mathbb E\left[ 2 \int_{t}^{T} e^{\beta (A_s+ s)} \bar{y}_s\bar{f}_s d A_s + 2 \int_{t}^{T} e^{\beta (A_s+ s)} \bar{y}_s\bar{g}_s d s\right].
\end{aligned}
\end{equation}

For $\bar Y$, denote the norm
$$
\|\bar Y\|_{*,\beta, A} :=
\left(\mathbb{E}\left[\int_0^T e^{\beta (A_s+ s)} \bar{Y}_s^2 d A_s\right]\right)^{1 / 2}.
$$
Correspondingly,
$$
\|\bar{y}\|_{*,\beta, A,[T-h, T]} := \left(\mathbb{E}\left[\int_{T-h}^T e^{\beta (A_s+ s)} \bar{y}_s^2 d A_s\right]\right)^{1 / 2}.
$$
Define the norms $\|\cdot\|_{*, \beta, p}$, $\|\cdot\|_{*, \beta, W}$, $\|\cdot\|_{*, \beta, p,[T-h, T]}$ and $\|\cdot\|_{*, \beta, W,[T-h, T]}$ in the similar manner.

Moreover, define:
$$
\|\bar{y}\|_{S^2_*,[T-h, T]} := \left(\sup _{T-h \leq s \leq T} \mathbb{E}\left[ \bar{y}_s^2 \right]\right)^{1 / 2}.
$$
Using the Lipschitz properties of $f$ and $g$, it turns out that,
$$
\begin{aligned}
& \beta\|\bar{y}\|_{*, \beta, A,[T-h, T]}^2+\beta\|\bar{y}\|_{*, \beta, W,[T-h, T]}^2+\|\bar{Z}\|_{*, \beta, W,[T-h, T]}^2 +\|\bar{U}\|_{*, \beta, p,[T-h, T]}^2\\
& \leq 2 L_f \mathbb{E}\left[\int_{T-h}^T e^{\beta (A_s+ s)}\left|\bar{y}_s \| \bar{P}_s\right| d A_s\right]+2 L_p \mathbb{E}\left[\int_{T-h}^T e^{\beta (A_s+ s)}\left|\bar{y}_s\right|\left(\int_E\left|\bar{Q}_s^2 \phi_s(d e)\right|\right)^{1 / 2} d A_s\right] \\
& \quad+2 L_g \mathbb{E}\left[\int_{T-h}^T e^{\beta (A_s+ s)}\left|\bar{y}_s \| \bar{P}_s\right| d s\right]+2 L_w \mathbb{E}\left[\int_{T-h}^T e^{\beta (A_s+ s)}\left|\bar{y}_s \| \bar{R}_s\right| d s\right] .
\end{aligned}
$$
With the help of the inequality $2 a b \leq \alpha a^2+b^2 / \alpha$, $a, b \geq 0$, we obtain:
$$
\begin{aligned}
& \beta\|\bar{y}\|_{*, \beta, A,[T-h, T]}^2+\beta\|\bar{y}\|_{*, \beta, W,[T-h, T]}^2+\|\bar{Z}\|_{*, \beta, W,[T-h, T]}^2 +\|\bar{U}\|_{*, \beta, p,[T-h, T]}^2\\
& \leq \frac{L_f}{\sqrt{\alpha}}\|\bar{y}\|_{*, \beta, A,[T-h, T]}^2+L_f \sqrt{\alpha}\|\bar{P}\|_{*, \beta, A,[T-h, T]}^2+\frac{L_p^2}{\alpha}\|\bar{y}\|_{*, \beta, A,[T-h, T]}^2+\alpha\|\bar{Q}\|_{*, \beta, p,[T-h, T]}^2 \\
& \quad+\frac{L_g}{\sqrt{\alpha}}\|\bar{y}\|_{*, \beta, W,[T-h, T]}^2+L_g \sqrt{\alpha}\|\bar{P}\|_{*, \beta, W,[T-h, T]}^2+\frac{L_w^2}{\alpha}\|\bar{y}\|_{*, \beta, W,[T-h, T]}^2+\alpha\|\bar{R}\|_{*, \beta, W,[T-h, T]}^2 
\end{aligned}
$$
for any $\alpha>0$.

Recall Lemma \ref{A priori estimate}, we have $\|\bar{y}\|_{S_*^2} < \infty$. Then, in view of (\ref{eq ito}) we can deduce that:
$$
\begin{aligned}
\|\bar{y}\|_{S^2_*,[T-h, T]}^2 &= \sup _{T-h \leq s \leq T} \mathbb{E}\left[  \bar{y}_s^2 \right] \\
&\leq \sup _{T-h \leq s \leq T}\mathbb{E}\left[  e^{\beta (A_s+ s)} \bar{y}_s^2 \right]\\
&\le 2 \mathbb{E}\left[\int_{T-h}^{T} e^{\beta (A_s+ s)} |\bar{y}_s\bar{f}_s| d A_s\right]+2 \mathbb{E}\left[\int_{T-h}^{T} e^{\beta (A_s+ s)} |\bar{y}_s\bar{g}_s| d s\right]\\
& \leq \frac{L_f}{\sqrt{\alpha}}\|\bar{y}\|_{*, \beta, A,[T-h, T]}^2+L_f \sqrt{\alpha}\|\bar{P}\|_{*, \beta, A,[T-h, T]}^2+\frac{L_p^2}{\alpha}\|\bar{y}\|_{*, \beta, A,[T-h, T]}^2+\alpha\|\bar{Q}\|_{*, \beta, p,[T-h, T]}^2 \\
& \quad+\frac{L_g}{\sqrt{\alpha}}\|\bar{y}\|_{*, \beta, W,[T-h, T]}^2+L_g \sqrt{\alpha}\|\bar{P}\|_{*, \beta, W,[T-h, T]}^2+\frac{L_w^2}{\alpha}\|\bar{y}\|_{*, \beta, W,[T-h, T]}^2+\alpha\|\bar{R}\|_{*, \beta, W,[T-h, T]}^2 .
\end{aligned}
$$
Add the two inequalities we have:
$$
\begin{aligned}
& \|\bar{y}\|_{S^2_*,[T-h, T]}^2 + \beta\|\bar{y}\|_{*, \beta, A,[T-h, T]}^2+\beta\|\bar{y}\|_{*, \beta, W,[T-h, T]}^2+\|\bar{Z}\|_{*, \beta, W,[T-h, T]}^2+\|\bar{U}\|_{*, \beta, p,[T-h, T]}^2\\
& \leq 2 \frac{L_f}{\sqrt{\alpha}}\|\bar{y}\|_{*, \beta, A,[T-h, T]}^2+2 L_f \sqrt{\alpha}\|\bar{P}\|_{*, \beta, A,[T-h, T]}^2+2 \frac{L_p^2}{\alpha}\|\bar{y}\|_{*, \beta, A,[T-h, T]}^2+2 \alpha\|\bar{Q}\|_{*, \beta, p,[T-h, T]}^2 \\
& \quad+2 \frac{L_g}{\sqrt{\alpha}}\|\bar{y}\|_{*, \beta, W,[T-h, T]}^2+2 L_g \sqrt{\alpha}\|\bar{P}\|_{*, \beta, W,[T-h, T]}^2+2 \frac{L_w^2}{\alpha}\|\bar{y}\|_{*, \beta, W,[T-h, T]}^2+2 \alpha\|\bar{R}\|_{*, \beta, W,[T-h, T]}^2 .
\end{aligned}
$$
Rearranging the terms, we find
\begin{equation}
\label{eq relation}
\begin{aligned}
& \|\bar{y}\|_{S^2_*,[T-h, T]}^2 + \|\bar{U}\|_{*, \beta, p,[T-h, T]}^2+\|\bar{Z}\|_{*, \beta, W,[T-h, T]}^2\\
& \quad +\left(\beta-\frac{2 L_p^2}{\alpha}-\frac{2 L_f}{\sqrt{\alpha}}\right)\|\bar{y}\|_{*, \beta, A,[T-h, T]}^2  +\left(\beta-\frac{2 L_W^2}{\alpha}-\frac{2 L_g}{\sqrt{\alpha}}\right)\|\bar{y}\|_{*, \beta, W,[T-h, T]}^2 \\
& \quad \leq 2 L_f \sqrt{\alpha}\|\bar{P}\|_{*, \beta, A,[T-h, T]}^2+2 \alpha\|\bar{Q}\|_{*, \beta, p,[T-h, T]}^2+2 L_g \sqrt{\alpha}\|\bar{P}\|_{*, \beta, W,[T-h, T]}^2+2 \alpha\|\bar{R}\|_{*, \beta, W,[T-h, T]}^2 .
\end{aligned}
\end{equation}
Since $\beta > 256 \kappa^4[3(L_f + L_g) + 2(L_p^2 + L_w^2)]>2 L_p^2+3 L_f+2L_w^2+3L_g$, then as long as we are able to choose a constant $0<\alpha<1$, we have
\begin{equation}
\label{*}
    \beta>\frac{2 L_p^2}{\alpha}+\frac{3 L_f}{\sqrt{\alpha}}+\frac{2 L_w^2}{\alpha}+\frac{3 L_g}{\sqrt{\alpha}}.
\end{equation}
The exact form of $\alpha$ will be determined later. The relation (\ref{eq relation}) rewrites as
$$
\begin{aligned}
& \|\bar{y}\|_{S^2_*,[T-h, T]}^2 + \frac{L_f}{\sqrt{\alpha}}\|\bar{y}\|_{*, \beta, A,[T-h, T]}^2+\frac{L_g}{\sqrt{\alpha}}\|\bar{y}\|_{*, \beta, W,[T-h, T]}^2+\|\bar{U}\|_{*, \beta, p,[T-h, T]}^2+\|\bar{Z}\|_{*, \beta, W,[T-h, T]}^2 \\
& \quad \leq 2 L_f \sqrt{\alpha}\|\bar{P}\|_{*, \beta, A,[T-h, T]}^2+2 \alpha\|\bar{Q}\|_{*, \beta, p,[T-h, T]}^2+2 L_g \sqrt{\alpha}\|\bar{P}\|_{*, \beta, W,[T-h, T]}^2+2 \alpha\|\bar{R}\|_{*, \beta, W,[T-h, T]}^2 \\
& \quad=2 \alpha\left(\frac{L_f}{\sqrt{\alpha}}\|\bar{P}\|_{*, \beta, A,[T-h, T]}^2+\|\bar{Q}\|_{*, \beta, p,[T-h, T]}^2+\frac{L_g}{\sqrt{\alpha}}\|\bar{P}\|_{*, \beta, W,[T-h, T]}^2+\|\bar{R}\|_{*, \beta, W,[T-h, T]}^2\right) .
\end{aligned}
$$
By the definition of the norms, Remark \ref{remark1}, Remark \ref{remark bdd} and (\ref{eq 8}),  we can conclude that,
$$
\begin{aligned}
\|\bar{Y}_s\|^2_{*, \beta, A,[T-h, T]} &= \left\|\bar{y}_s + \sup _{T-h \leq s \leq T} L_s(y_s^1) - \sup _{T-h \leq s \leq T} L_s(y_s^2)\right\|^2_{*, \beta, A,[T-h, T]}\\
&\leq \mathbb{E}\left[\int_{T-h}^{T} e^{\beta (A_s+ s)} \left(\bar{y}_s+ \sup _{T-h \leq s \leq T} |L_s(y_s^1)- L_s(y_s^1)|\right)^2 d A_s\right]\\
&\leq \mathbb{E}\left[\int_{T-h}^{T} e^{\beta (A_s+ s)} \left(\bar{y}_s+ \kappa \sup _{T-h \leq s \leq T} \mathbb{E}[|\bar{y}_s|]\right)^2 d A_s\right]\\
&\leq  \mathbb{E}\left[\int_{T-h}^{T} e^{\beta (A_s+ s)} \left(2 \bar{y}_s^2 + 2 \kappa^2 \sup _{T-h \leq s \leq T} \mathbb{E}[\bar{y}_s^2]\right) d A_s\right]\\
&\leq  2\|\bar{y}\|^2_{*, \beta, A,[T-h, T]} + 2 \kappa^2 \mathbb{E}\left[\int_{T-h}^{T} e^{\beta (A_s+ s)} d A_s\right] \|\bar{y}\|_{S^2_*,[T-h, T]}^2.
\end{aligned}
$$
In the same way, we also have: 
$$
\begin{aligned}
& \|\bar{Y}_s\|^2_{*, \beta, W,[T-h, T]} \leq 2 \|\bar{y}\|^2_{*, \beta, W,[T-h, T]} + 2 \kappa^2 \mathbb{E}\left[\int_{T-h}^{T} e^{\beta (A_s+ s)} d s\right] \|\bar{y}\|_{S^2_*,[T-h, T]}^2,\\
& \|\bar{Y}_s\|_{S^2_*,[T-h, T]}^2 \leq (2+2\kappa^2)\|\bar{y}\|_{S^2_*,[T-h, T]}^2.
\end{aligned}
$$
Hence, we conclude that:
\begin{equation}
\begin{aligned}
    & \|\bar{Y}\|_{S^2_*,[T-h, T]}^2 + \frac{L_f}{\sqrt{\alpha}}\|\bar{Y}\|_{*, \beta, A,[T-h, T]}^2+\frac{L_g}{\sqrt{\alpha}}\|\bar{Y}\|_{*, \beta, W,[T-h, T]}^2+\|\bar{U}\|_{*, \beta, p,[T-h, T]}^2+\|\bar{Z}\|_{*, \beta, W,[T-h, T]}^2 \\ 
    &\leq  (2+2\kappa^2)\|\bar{y}\|_{S^2_*,[T-h, T]}^2+\frac{L_f}{\sqrt{\alpha}} \left[ 2\|\bar{y}\|^2_{*, \beta, A,[T-h, T]} + 2 \kappa^2 \mathbb{E}\left[\int_{T-h}^{T} e^{\beta (A_s+ s)} d A_s\right] \|\bar{y}^2\|_{S^2_*,[T-h, T]}^2 \right]\\
    &\quad + \frac{L_g}{\sqrt{\alpha}} \left[ 2 \|\bar{y}\|^2_{*, \beta, W,[T-h, T]} + 2 \kappa^2 \mathbb{E}\left[\int_{T-h}^{T} e^{\beta (A_s+ s)} d s\right] \|\bar{y}\|_{S^2_*,[T-h, T]}^2 \right] \\
    &\quad + \|\bar{U}\|_{*, \beta, p,[T-h, T]}^2+\|\bar{Z}\|_{*, \beta, W,[T-h, T]}^2 \\
    &\leq 4 \alpha \left[ 2 \kappa^2 \left( \frac{L_f}{\sqrt{\alpha}} + \frac{L_g}{\sqrt{\alpha}}\right)\mathbb{E}\left[ \int_{T-h}^{T} e^{\beta (A_s+ s)} (d A_s+ d s)\right] +2 \kappa ^2 \right]\\
    &\quad \times \left(\frac{L_f}{\sqrt{\alpha}}\|\bar{P}\|_{*, \beta, A,[T-h, T]}^2+\|\bar{Q}\|_{*, \beta, p,[T-h, T]}^2+\frac{L_g}{\sqrt{\alpha}}\|\bar{P}\|_{*, \beta, W,[T-h, T]}^2+\|\bar{R}\|_{*, \beta, W,[T-h, T]}^2\right).
\end{aligned}
\end{equation}
The last inequality holds since $\kappa>1$.

Since $\beta > 256 \kappa^4[3(L_f + L_g) + 2(L_p^2 + L_w^2)]$, with the help of assumption ($\mathbf{V}$) and  dominated convergence theorem, we are allowed to choose a small enough $h$ such that: $T=nh$, and
\begin{equation}
\label{eq condition}
    \beta > 256\kappa^4 [3(L_f + L_g) + 2(L_p^2 + L_w^2)]\left[\left( L_f+ L_g \right)\max_{1\le j\le n}\mathbb{E}\left[ \int_{(j-1)h}^{jh} e^{\beta (A_s+ s)} (d A_s+ d s)\right]+1 \right]^2.
\end{equation}
Let
$$
\alpha :=\frac{1}{256\kappa^4\left[ 1+\left( L_f+ L_g \right)\mathbb{E}\left[ \int_{T-h}^{T} e^{\beta (A_s+ s)} (d A_s+ d s)\right]\right]^2} <1.
$$
Then (\ref{*}) holds.
Meanwhile,
$$
4 \alpha \left[ 2 \kappa^2 \left( \frac{L_f}{\sqrt{\alpha}} + \frac{L_g}{\sqrt{\alpha}}\right)\mathbb{E}\left[ \int_{T-h}^{T} e^{\beta (A_s+ s)} (d A_s+ d s)\right] +2 \kappa ^2 \right] < \frac{1}{2} < 1.
$$
Thus we can deduce that 
\begin{equation}
\begin{aligned}
    & \|\bar{Y}\|_{S^2_*,[T-h, T]}^2 + \frac{L_f}{\sqrt{\alpha}}\|\bar{Y}\|_{*, \beta, A,[T-h, T]}^2+\|\bar{U}\|_{*, \beta, p,[T-h, T]}^2+\frac{L_g}{\sqrt{\alpha}}\|\bar{Y}\|_{*, \beta, W,[T-h, T]}^2+\|\bar{Z}\|_{*, \beta, W,[T-h, T]}^2 \\ 
    &\leq \frac{1}{2} \left(\frac{L_f}{\sqrt{\alpha}}\|\bar{P}\|_{*, \beta, A,[T-h, T]}^2+\|\bar{Q}\|_{*, \beta, p,[T-h, T]}^2+\frac{L_g}{\sqrt{\alpha}}\|\bar{P}\|_{*, \beta, W,[T-h, T]}^2+\|\bar{R}\|_{*, \beta, W,[T-h, T]}^2\right).
\end{aligned}
\end{equation}
Furthermore,
\begin{equation}
\label{eq condition3}
\begin{aligned}
    & \frac{L_f}{\sqrt{\alpha}}\|\bar{Y}\|_{*, \beta, A,[T-h, T]}^2+\|\bar{U}\|_{*, \beta, p,[T-h, T]}^2+\frac{L_g}{\sqrt{\alpha}}\|\bar{Y}\|_{*, \beta, W,[T-h, T]}^2+\|\bar{Z}\|_{*, \beta, W,[T-h, T]}^2 \\ 
    &\leq \frac{1}{2} \left(\frac{L_f}{\sqrt{\alpha}}\|\bar{P}\|_{*, \beta, A,[T-h, T]}^2+\|\bar{Q}\|_{*, \beta, p,[T-h, T]}^2+\frac{L_g}{\sqrt{\alpha}}\|\bar{P}\|_{*, \beta, W,[T-h, T]}^2+\|\bar{R}\|_{*, \beta, W,[T-h, T]}^2\right).
\end{aligned}
\end{equation}

Therefore, $\Gamma$ defines a strict contraction map on the time interval $[T-h,T]$ with respect to the equivalent norm to $\|\cdot\|_{\mathbb L^{\beta}}$,
$$
\|(Y, U, Z)\|_{\mathbb{L}^{\beta},*,\alpha}^2:=\frac{L_f}{\sqrt{\alpha}}\|Y\|_{*, \beta, A}^2+\|{U}\|_{*, \beta, p}^2+\frac{L_g}{\sqrt{\alpha}}\|Y\|_{*, \beta, W}^2+\|{Z}\|_{*, \beta, W}^2,
$$
which implies the existence and uniqueness of local solution on $[T-h,T]$ due to the completeness of the space. 

Next we stitch the local solutions to get the global solution on $[0,T]$. More precisely, choose  $h$ such that (\ref{eq condition}) holds and $T=n h$ for some integer $n$. Then  the BSDE (\ref{eq3}) with mean reflection admits a unique deterministic flat solution $\left(Y^n, U^n, Z^n, K^n\right) \in \mathbb{L}^\beta_{[T-h, T]} \times \mathcal{A}_D(T-h, T)$ on the time interval $[T-h, T]$. Next we take $T-h$ as the terminal time and $Y_{T-h}^n$ as the terminal condition. We can find the unique deterministic flat solution of the BSDE (\ref{eq3}) with mean reflection $\left(Y^{n-1}, U^{n-1}, Z^{n-1}, K^{n-1}\right)$  on the time interval $[T-2 h, T-h]$. 
Repeating this procedure, we get a sequence $\left(Y^i, U^i, Z^i, K^i\right)_{i \leq n}$. This procedure works since for $\beta$ satisfying (\ref{eq condition}) and $h$ chosen before, for each $1\le i\le n$, 
 \begin{equation}
    \beta > 256\kappa^4 [3(L_f + L_g) + 2(L_p^2 + L_w^2)]\left[\left( L_f+ L_g \right)\mathbb{E}\left[ \int_{(i-1)h}^{ih} e^{\beta (A_s+ s)} (d A_s+ d s)\right]+1 \right]^2.
\end{equation}
It turns out that on $[(i-1)h,ih]$, we are able to choose $\alpha^{(i)}$ such that,
\begin{equation}
\begin{aligned}
    &\frac{L_f}{\sqrt{\alpha^{(i)}}}\|\bar{Y}\|_{*, \beta, A,[(i-1)h, ih]}^2+\|\bar{U}\|_{*, \beta, p,[(i-1)h, ih]}^2+\frac{L_g}{\sqrt{\alpha^{(i)}}}\|\bar{Y}\|_{*, \beta, W,[(i-1)h, ih]}^2+\|\bar{Z}\|_{*, \beta, W,[(i-1)h, ih]}^2 \\ 
    &\leq \frac{1}{2} \left(\frac{L_f}{\sqrt{\alpha^{(i)}}}\|\bar{P}\|_{*, \beta, A,[(i-1)h, ih]}^2+\|\bar{Q}\|_{*, \beta, p,[(i-1)h, ih]}^2+\frac{L_g}{\sqrt{\alpha^{(i)}}}\|\bar{P}\|_{*, \beta, W,[(i-1)h, ih]}^2+\|\bar{R}\|_{*, \beta, W,[(i-1)h, ih]}^2\right),
\end{aligned}
\end{equation}
which implies the local well-posedness on $[(i-1)h,ih]$.

We stitch the sequence as:
$$
Y_t=\sum_{i=1}^n Y_t^i I_{[(i-1) h, i h)}(t)+Y_T^n I_{\{T\}}(t), \ U_t=\sum_{i=1}^n U_t^i I_{[(i-1) h, i h)}(t)+U_T^n I_{\{T\}}(t), \ Z_t=\sum_{i=1}^n Z_t^i I_{[(i-1) h, i h)}(t)+Z_T^n I_{\{T\}}(t),
$$
as well as
$$
K_t=K_t^i+\sum_{j=1}^{i-1} K_{j h}^j, \quad \text { for } t \in[(i-1) h, i h],  i \leq n.
$$
 It is  obvious that $(Y, U, Z, K) \in L^{2, \beta}(A) \cap$ $L^{2, \beta}(W) \times L^{2, \beta}(p) \times L^{2, \beta}(W) \times \mathcal{A}_D$ is a deterministic flat solution to the BSDE (\ref{eq3}) with mean reflection.
The uniqueness of the global solution follows from the uniqueness of local solution on each small time interval. The proof is complete.

\end{proof}

\section{Application}
\label{sec app}
BSDEs with mean reflection provide a tool for super-hedging under a running risk constraint, as in \cite{briand2018bsdes}. Inspired by \cite[\S 7.2]{delong2013backward}, in this section, we present a more general example in which a mean reflected BSDE driven by a marked point process is used for super-hedging in the financial and insurance market.

In order to measure the risk of a portfolio, the so-called risk measures are  introduced, see e.g. \cite{artzner1999coherent}. For a fixed $t$, define a static risk measure as a map $\rho(t,):. L^{2, \beta}(A) \cap L^{2, \beta}(W) \longrightarrow \mathbb{R}$ satisfying $\rho(t, 0)=0$ together with
\begin{itemize}
    \item Monotonicity: $X \leq Y \Longrightarrow \rho(t, X) \geq \rho(t, Y)$, for $X, Y \in L^{2, \beta}(A) \cap L^{2, \beta}(W)$;
    \item Translation invariance: $\rho(t, X+m)=\rho(t, X)-m$, for $X \in L^{2, \beta}(A) \cap L^{2, \beta}(W)$ and $m \in \mathbb{R}$.
\end{itemize}

Besides, risk measures can similarly be characterized by an acceptance set, which is defined as
$$
\mathcal{A}_\rho^t=\left\{X \in L^{2, \beta}(A) \cap L^{2, \beta}(W): \rho(t, X) \leq 0\right\}.
$$
Given a set $\mathcal{A}^t$,
$$
\rho(t, X)=\inf \left\{m \in \mathbb{R}: m+X \in \mathcal{A}^t\right\}
$$
is a static risk measure with $\mathcal A^t=\mathcal A^t_{\rho}$. Notice that, the acceptance set $\mathcal{A}^t$ and the risk measure $\rho(t,.)$ share a one to one correspondence. As in \cite{briand2018bsdes}, for a given collection of static risk measures $(\rho(t, .))_t$, a wealth process $Y$ is admissible  as soon as it satisfies
\begin{equation}
\label{eq risk constraint}
    \rho\left(t, Y_t\right) \leq c_t, \quad 0 \leq t \leq T,
\end{equation}
where $c$ is a given time indexed deterministic benchmark. Consistently, the Skorohod type condition is generalized to
$$
\int_0^T\left[c_t-\rho\left(t, Y_t\right)\right] d K_t=0.
$$
In view of \cite[Theorem 13]{briand2018bsdes}, the following  similar theorem enables us to consider BSDEs under risk measure constraint of the form (\ref{eq risk constraint}).
\begin{thm}
\label{thm3}
  Let $\rho(t,):[0, T] \times L^{2, \beta}(A) \cap L^{2, \beta}(W) \longrightarrow \mathbb{R}$ be a collection of monotonic and translation invariant risk measures, which are continuous with time and Lipschitz in space, i.e.
$$
|\rho(t, X)-\rho(t, Y)| \leq \kappa \mathbb{E}[|X-Y|], \quad 0 \leq t \leq T, \quad X, Y \in L^{2, \beta}(A) \cap L^{2, \beta}(W).
$$
Moreover, given a continuous deterministic benchmark $c$ and $\xi$ satisfying $\rho(T, \xi) \leq c_T$. Let assumptions $(\mathbf{I})$, $\left(\mathbf{II}\right)$ and $\left(\mathbf{V}\right)$ hold for some $\beta$ satisfying  $\beta > 256 \kappa^4 [3(L_f + L_g) + 2(L_p^2 + L_w^2)]$, then the "BSDE with risk measure reflection"
\begin{equation}
\label{eq BSDE with risk measure}
\left\{\begin{array}{l}
Y_t =\xi+\int_t^T f\left(s, Y_s, U_s\right) d A_s+\int_t^T g\left(s, Y_s, Z_s\right) d s\\
\quad \quad-\int_t^T \int_E U_s(e) q(d s d e)-\int_t^T Z_s d W_s + \left(K_T-K_t\right) , \quad \forall t \in[0, T] \text { a.s. };\\
\rho\left(t, Y_t\right) \leq c_t, \quad 0 \leq t \leq T, \quad \int_0^T\left[c_t-\rho\left(t, Y_t\right)\right] d K_t=0 ,
\end{array}\right.
\end{equation}
admits a unique deterministic flat solution.
\end{thm}

\begin{proof}
    Replace the map $L_t$  by $\rho(t,.)-c_t$, for any $t \in[0, T]$ and then we can follow the proof above to get the well-posedness of (\ref{eq BSDE with risk measure}). 
\end{proof}

For the sake of financial applications, a typical choice of $\rho$ which is Lipschitz with respect to $X$,  is the classical Expected Shortfall risk measure defined as
$$
\rho_\alpha^{E S}(t, X):=\frac{1}{\alpha_t} \int_0^{\alpha_t} \operatorname{VaR}_s(X) d s,
$$
where $\alpha_t \in(0,1)$ denotes a given precision level and $V a R_s$ is the Value at Risk of level $s$.

Next we discuss a super-hedging problem in the financial and insurance market. Assume the dynamics of the bank account $S^0:=\left\{S^0_t\right\}_{ 0 \leq t \leq T}$ is
\begin{equation}
\label{eq bank}
    \frac{d S^0_t}{S^0_t}=r_t d t, \quad S^0_0=1,
\end{equation}
where $r:=\{r_t\}_{0 \leq t \leq T}$ is deterministic and denotes the risk-free rate. The stock price $S:=\left\{S_t\right\}_{ 0 \leq t \leq T}$ is described by the following geometric Brownian motion,
\begin{equation}
\label{eq stock}
    \frac{d S_t}{S_t}=\mu_t d t+\sigma_t d W_t, \quad S_0=s>0,
\end{equation}
where $\mu:=\left\{\mu_t\right\}_{ 0 \leq t \leq T}$ denotes the expected return on the stock and $\sigma:=\left\{\sigma_t\right\}_{0 \leq t \leq T}$ denotes the stock volatility. We assume that the drift $\mu$ and the volatility $\sigma$ are bounded predictable processes.  For the sake of market completeness, assume in addition that $\sigma_t \sigma_t^{\prime}-\varepsilon I \succeq 0$ for some $\epsilon>0$. Consider the following life insurance portfolio consisting of $n$ persons insured whose death is triggered by a marked point process  $p$, with compensator  $\nu(dtde)=(n-N_{t^-})\lambda_t(e)\tilde p(e)\delta_{\{e\}}(de)dt$. 
Here $\lambda(\cdot):[0, T] \rightarrow(0, \infty)$ denotes a (deterministic) mortality intensity.
 We attribute the death of people to different causes such as  natural death, traffic accident, sudden illness and etc. Then each death is marked by an element in the mark space $E$, which is assumed finite. Each situation $e\in E$ occurs with probability $\tilde p(e)$.
The running cash flow of the life insurance portfolio reads:
\begin{equation}
\label{eq insurance}
\begin{aligned}
P_t= & \int_0^t(n-N_s) {H}_s d s+\int_0^t\int_{E}{G}_s(e)p(dsde), \quad 0 \leq t \leq T,
\end{aligned}
\end{equation}
in which  $H$ is a deterministic process representing the insurance premium received continuously during the period of the contract, and $G$ is a deterministic death benefit paid at random times triggered by the marked point process $p$. Moreover, the corresponding counting process $N$ takes the form,
$$
N_t=\sum_{i=1}^n \mathbf{1}\left\{\tau_i \leq t\right\}, \quad 0 \leq t \leq T,
$$
where $\left(\tau_i, i=1, \ldots, n\right)$ is a sequence of random variables which are, conditional on the filtration $\mathscr{F}$, independent and exponentially distributed
$$
\mathbb{P}\left(\tau_i>t \mid \mathscr{F}_t\right)=e^{-\int_0^t\sum_{e\in E}\tilde p(e)\lambda_s(e) d s}, \quad i=1, \ldots, n.
$$
Properties of the point process $N$ are studied by Jeanblanc and Rutkowski \cite{Jeanblanc2002} in a credit risk context and Dahl and Møller \cite{DAHL2006193} in a life insurance context.

Next, we consider an insurance company who faces the cash flow generated by the insurance (\ref{eq insurance}) and invests in the bank account (\ref{eq bank}) and the stock (\ref{eq stock}). To hedge the  mortality risk and build a super-hedging portfolio, they also invests mortality bond which pays a unit for each insured person who survives till the maturity of the contract. 
For a given initial capital $x$, we consider portfolios $X^{x, \pi,\chi, K}$ driven by a consumption-investment strategy $(\pi,\chi, K)$, and whose dynamics is given by
$$
\begin{aligned}
d X_t^{x, \pi,\chi, K} & = \pi_t(\mu_t d t+\sigma_t d W_t)+\chi_t\frac{dD_t}{D_{t^-}}+\left(X_t^{x, \pi,\chi, K}-\pi_t-\chi_t\right) r_t d t-dP_t-d K_t
\end{aligned}
$$
where $\pi$ and $\chi$ denote the amount of wealth invested in the stock $S$ and in the mortality bond $D$, respectively. Here $K$ is a dynamic risk premium due to a running risk constraint.

Consider pricing measure $\mathbb Q$ which satisfies:
$$
\mathbb{Q} \sim \mathbb{P},\ \frac{d \mathbb{Q}}{d \mathbb{P}} \mid \mathscr{F}_t=M_t,\ 0 \leq t \leq T,\
M \text { is a positive } \mathscr{F} \text {-martingale satisfying},
$$
$$
\frac{d M_t}{M_{t^-}}= -\theta_t d W_t+\int_E \kappa_t(e) q(dtde), \quad M(0)=1,
$$
where $\theta_t=\frac{\mu_t-r_t}{\sigma_t}$, and  $\kappa$ is a predictable process such that
$$
\begin{aligned}
&\mathbb E\left[ \int_0^T \sum_{e\in E}|\kappa_t(e)|^2 \lambda_t(e)\tilde p(e)d t\right]<\infty, \\
& \kappa_t(e)>-1, \quad\forall (t, e) \in[0, T] \times E .
\end{aligned}
$$
The processes $\kappa$ is called the market price of the insurance risk or the risk premium required by investors for taking, respectively, the unsystematic insurance risk.

Then, making use of a similar argument as \cite[Proposition 9.4.1]{delong2013backward}, under the additional assumptions that $\lambda$ is continuous in $t$ and $\kappa$ is bounded, the price of the mortality bond, $D_t=\mathbb{E}^{\mathbb{Q}}\left[e^{-\int_t^T r_s ds}(n-N_T) \mid \mathscr{F}_t\right], \quad 0 \leq t \leq T$,  satisfies the dynamics,
$$
\begin{aligned}
d D_t= D_{t^-} \mathbf{1}\{n-N_{t^-}>0\}\left(\left(r_t+\sum_{e\in E}(1+\kappa_t(e)) \lambda_t(e)\tilde p(e)\right) d t-\frac{1}{n-N_{t^-}} d N_t\right), \quad 0 \leq t \leq T.
\end{aligned}
$$
Then, 
\begin{equation}
\label{wealth}
\begin{aligned}
d X^{x,\pi, \chi, K}_t= & \pi_t \frac{d S_t}{S_t}+\chi_t \frac{d D_t}{D_{t^-}}+\left(X^{x,\pi, \chi, K}_{t}-\pi_t-\chi_t\right) \frac{d S^0_t}{S^0_t}-dP_t-dK_t \\
= & \pi_t(\mu_t d t+\sigma_t d W_t) +\chi_t \mathbf{1}\{n-N_{t^-}>0\}\left(\left(r_t+\sum_{e\in E} (1+\kappa_t(e)) \lambda_t(e)\tilde p(e)\right) d t-\frac{1}{n-N_{t^-}} d N_t\right)\\
& +\left(X^{x,\pi, \chi,K}_{t^-}-\pi_t-\chi_t \mathbf{1}\{n-N_{t^-}>0\}\right) r_td t-dP_t-dK_t\\
=& \left(r_tX^{x,\pi, \chi, K}_{t}+\pi_t\mu_t-\pi_tr_t-(n-N_t)H_t+\chi_t \mathbf{1}\{n-N_{t^-}>0\}\sum_{e\in E} (1+\kappa_t(e)) \lambda_t(e)\tilde p(e)\right)dt \\
&\quad -\int_E\left(\frac{\chi_t}{n-N_{t^-}} \mathbf{1}\{n-N_{t^-}>0\}+G_t(e)\right)p(dtde)-dK_t,\quad X^{x,\pi, \chi,K}(0)=x .
\end{aligned}
\end{equation}

In this paper, a strategy $(\pi,\chi,K)$  is considered admissible if and only if it satisfies the following constraint:
$$
\rho_\alpha^{E S}\left(t, X_t^{x, \pi,\chi, K}\right) \leq c_t, \quad 0 \leq t \leq T,
$$
where $(\alpha, c)$ are deterministic quantile and level benchmarks. The collection of admissible policies is denoted by $\mathcal A$. The goal is looking for a super hedging price in the collection
$$
\left\{x \in \mathbb{R}, \quad \exists(\pi,\chi, K) \in \mathcal{A}, \quad \text { s.t. } \quad X_T^{x, \pi,\chi, K} \geq  (n-N_T) {F} \quad \text { and } \quad \rho_\alpha^{E S}\left(t, X_t\right) \leq c_t, \quad \forall t \in[0, T]\right\},
$$
and associated consumption-investment strategy, where the random variable $F$ is a survival benefit paid at the end of the contract. Applying the results of this paper, we deduce that, taking only deterministic $K$ strategies into consideration, an admissible super-hedging price $Y_0$ is well defined as the starting point of the unique deterministic flat solution to the following BSDE with risk measure reflection
\begin{equation}
\left\{\begin{array}{l}
Y_t= (n-N_T) F
+\int_t^T\left(-Y_s r_s-\theta_s Z_s+H_s(n-N_{s})+(n-N_{s^-})\sum_{e\in E}G_s(e)(1+\kappa_s(e))\lambda_s(e)\tilde p(e)\right.\\
\quad\quad\quad\quad\quad\quad\quad\quad\quad\left.+(n-N_{s^-})\sum_{e\in E}  U_s(e) \kappa_s(e) \lambda_s(e)\tilde p(e)
\right) d s \\
\quad\quad\quad-\int_t^T Z_s d W_s-\int_t^T\int_{E} U_s(e)q(dsde)+K_T-K_t, \quad 0 \leq t \leq T,\\
\rho_\alpha^{E S}\left(t, Y_t\right) \leq c_t, \quad 0 \leq t \leq T, \quad \int_0^T\left[c_t-\rho_\alpha^{E S}\left(t, Y_t\right)\right] d K_t=0.
\end{array}\right.
\end{equation}
which can be immediately derived from the wealth process (\ref{wealth}) by introducing the variables
$$
\begin{aligned}
& Y_t=X^{x,\pi, \chi,K}_t, \quad 0 \leq t \leq T, \\
& Z_t=\pi_t \sigma_t, \quad 0 \leq t \leq T, \\
& U_t(e)=\frac{-\chi_t}{n-N_{t^-}} \mathbf{1}\{n-N_{t^-}>0\}-G_t(e), \quad 0 \leq t \leq T,\ e\in E,
\end{aligned}
$$
where the equations represent "in the same equivalence class". Then, with proper integrability conditions on $F,\ G$ and $H$, this mean reflected BSDE satisfies all assumptions in Theorem \ref{thm3} and is well-posed.

\bibliographystyle{plain}
\bibliography{MPP_RBSDE}

\end{document}